\documentclass{birkjour}
 \newtheorem{theorem}{Theorem}[section]
 
 \newtheorem{lemma}[theorem]{Lemma}
 \newtheorem{proposition}[theorem]{Proposition}
 \theoremstyle{definition}
 \newtheorem{definition}[theorem]{Definition}
 \theoremstyle{remark}
 \newtheorem{remark}[theorem]{Remark}
 \newtheorem{example}{Example}
 \numberwithin{equation}{section}
\usepackage{hyperref}
\usepackage{mathrsfs}
\usepackage{amsmath,geometry,amssymb}
\usepackage{graphicx}

\begin{document}

%-------------------------------------------------------------------------
% editorial commands: to be inserted by the editorial office
%
%\firstpage{1} \volume{228} \Copyrightyear{2004} \DOI{003-0001}
%\seriesextraline{This is the Concrete Title of this Book\br H.E. R and S.T.C. W, Eds.}
%\firstpage{1}

\title[]
 {Warped product Quasi Bi-slant Submanifolds of Kaehler Manifolds}

%----------Author 1
\author[Lone]{Mehraj Ahmad Lone$^{*}$}

\address{%
	Department of Mathematics,\\National Institute of Technology Srinagar, \\
	190006, Kashmir, India.}

\email{mehrajlone@nitsri.ac.in}

\author[Majeed]{Prince Majeed}

\address{%
	Department of Mathematics,\\National Institute of Technology Srinagar, \\
	190006, Kashmir, India.}

\email{prince$_{-}$05phd19@nitsri.net}

%\thanks{This work was completed with the support of our
%----------Author 2

\subjclass{53C15, 53C40, 53C42}

\keywords{Hemi-slant submanifolds, Quasi bi-slant submanifolds, Warped products.}

%\date{March 30, 2017}

\begin{abstract}
In this paper, we introduce the notion of warped product quasi bi-slant submanifolds in Kaehler manifolds. We have shown that every warped product quasi bi-slant submanifold in a Kaehler manifold is either a Riemannian product or a warped product quasi hemi slant submanifold. Furthermore, we provide examples for  both cases. 
\end{abstract}

%%% ----------------------------------------------------------------------
\maketitle
%%% ----------------------------------------------------------------------

\section{\protect \bigskip Introduction} 	
\par  Chen\cite{chen2005} introduced the  notion of slant submanifolds, the initial findings on slant submanifolds were collected in his book \cite{chen book}. Numerous geometer groups continue to study and conduct research on this idea of submanifolds. Recently, the related literature of slant submanifolds has been compiled in the form of two books by Chen, Shahid and Solamy (see \cite{chen a,chen b}). During last decade,  many generalizations and extensions of slant submanifolds have been introduced, like: semi-slant, pointwise slant, hemi-slant, pointwise hemislant and many more. The related literature of these kind of generalizations can be be found in (see, \cite{Garay,Etayo,VKhan,Lone,stonvic}).  A more generic class of submanifolds in the form of bi-slant submanifolds was introduced by Cabrerizo and Cariazo \cite{Cabrerizo}. This class of submanifolds acts as a natural generalization of CR, semi-slant,  slant,  hemi-slant submanifolds \cite{VKhan,Lone,papaghuic}. Further the extended notion of pointwise bi-slant submanifolds of Kaehler manifolds can be found in \cite{Uddin2}. Etayo\cite{Etayo} introduced the idea of pointwise slant submanifolds as an extension of slant submanifolds and gave them the label quasi-slant submanifolds. Prasad, Shukla, and Haseeb \cite{prasad2} recently proposed the notion of quasi hemi-slant submanifolds of Kaehler manifolds. This notion of quasi hemi-slant submanifolds was generalised by Prasad, Akyol, Verma, and Kumar \cite{Prasad1} to a more generic class of submanifolds in the form of quasi bi-slant submanifolds of Kaehler manifolds.
They established the prerequisites for the integrability of the distributions used in the definition of such submanifolds.

\par Bishop and O’Neill in $1960s$ introduced the concept of warped product manifolds. These manifolds find their applications both in physics as well as in Mathematics.  Since then the study of warped product submanifolds has been investigated by many geometers (see, \cite{solamy,ch,chnn,Chnnn}).  In particular,  Chen started looking these warped products as submanifolds of different kinds of manifolds (see, \cite{5,6}). In  Kaehlerian settings, he proved  besides CR- products the non-existence of warped products  of the form $N^\perp \times_f N^T$,  where $N^\perp,$ $N^T$  is a totally real  and holomorphic submanifold,  respectively.  Now from the past two decades this area of research is an active area of research among many of the geometers and theoretical physicists.  For the overall development of the subject we refer\cite{Chenslant} .\\

\par Now while importing the survey of warped products to  slant cases,  Sahin  in \cite{B.sahin} proved the non-existence of  semi-slant warped products in any Kaehler manifold. Then in \cite{B.sahinn} he extended the study to pointwise semi-slant warped products of Kaeherian manifolds. Uddin, Chen and Solamy \cite{Uddin} studied warped product bi-slant submanifolds in Kaehler manifolds. In this paper we have studied the notion of warped product quasi bi-slant submanifolds in Kaehler manifolds, we proved that every warped product quasi bi-slant submanifold in a Kaehler manifold is either a Riemannian product or a warped product quasi hemi slant submanifold. Moreover, we provide the examples of both the cases. 
 
\section{Preliminaries}
Let ($\bar{M},J,g$) be an almost Hermitian manifold with an almost complex structure $J$ and a Riemannian metric $g$ such that
\begin{align}\label{2.1}
	J^2 = -I,
\end{align} 
\begin{align}\label{2.2}
	g(JX,JY) = g(X,Y)
\end{align}
for any $X, Y \in \Gamma(T\bar{M})$, where $I$ is the identity map and $\Gamma(T\bar{M})$ denotes the set of all vector fields of $\bar{M}$. Let $\bar{\nabla}$ denotes the Levi-Civita connection on $\bar{M}$ with respect to the Riemannian metric $g$. If the almost complex structure $J$ satifies
\begin{align}\label{2.3}
	(\bar{\nabla}_XJ)Y = 0,
\end{align}
 for any vector $ X, Y \in \Gamma(T\bar{M})$, the $\bar{M}$ is called a Kaehler manifold.\\
Let $M$ be a Riemannian manifold isometrically immersed in $\bar{M}$ and we denote by the symbol $g$ the Riemannian metric induced on $M$. Let $\Gamma(TM)$ denote the Lie algebra of vector fields in $M$ and $\Gamma(T^\perp M)$, the set of all vector fields normal to $M$. If $\nabla$ be the induced Levi-Civita connection on $M$, the Gauss and Weingarten formulas are respectively given by 
\begin{align}\label{2.4}
	\bar{\nabla}_XY = \nabla_XY + \sigma(X,Y),
\end{align}
and 
\begin{align}\label{2.5}
	\bar{\nabla}_XN = -A_NX + \nabla^\perp_XN,
\end{align}
for any $ X,Y \in \Gamma(TM)$ and $N \in \Gamma(T^\perp M)$, where $\nabla^\perp$ is the normal connection on $T^\perp M$ and $A$ the shape operator. The shape operator and the second fundamental form of $M$ are related by
\begin{align}\label{2.6}
	g(A_NX,Y) = g(\sigma(X,Y),N),
\end{align}
for any $ X,Y \in \Gamma(TM)$ and $N \in \Gamma(T^\perp M)$, and $g$ denotes the induced metric on $M$ as well as the metric on $\bar{M}$.\\
For a tangent vector field $X$ and a normal vector field $N$ of $M$, we can write
\begin{align}\label{2.7}
	JX = \phi X + \omega X,
\end{align}
where $\phi X$ and $\omega X$ are the tangential and normal components of $JX$ on $M$ respectively. Similarly for $N \in \Gamma(T^{\perp}N)$, we have
\begin{align}\label{2.8}
	 JN = BN +CN,
\end{align}
where $BN$ and $CN$ are tangential and normal components of $JN$ on $M$ respectively. Moreover, from (\ref{2.2}), (\ref{2.7}) and (\ref{2.8}), we have
\begin{align}\label{2.9}
	g(TX,Y) = g(X,TY),
\end{align}
for any $X ,Y \in \Gamma(TM)$.\\
We can now specify the following classes of submanifolds of Hermitian manifolds for later use:\\
(1) A submanifold $M$ of an almost Hermitian manifold $\bar{M}$ is said to be slant (see \cite{chen2005}), if for each non-zero vector $X$ tangent to $M$, the angle $\theta(X)$ between $JX$ and $T_pM$ is a constant, i.e., it does not depend on the choice of $ p \in M$ and $ X \in T_pM$. In this case, the angle $\theta$ is called the slant angle of the submanifold. A slant submanifold $M$ is called proper slant submanifold if $\theta \neq 0, \frac{\pi}{2}$.\\
(2) A submanifold $M$ of an almost Hermitian manifold $\bar{M}$ is said to be invariant(holomorphic or complex) submanifold (see \cite{chen2005}), if $J(T_pM) \subseteq T_p(M)$ for every point $p \in M$.\\
(3) A submanifold $M$ of an almost Hermitian manifold $\bar{M}$ is said to be anti-invariant (totally real) submanifold (see \cite{KON}), if  $J(T_pM) \subseteq T^{\perp}_p(M)$ for every point $p \in M$.\\
(4) A submanifold $M$ of an almost Hermitian manifold $\bar{M}$ is said to be semi-invariant (see \cite{bejancu}), if there
exist two orthogonal complementary distributions $D$ and $D^{\perp}$ on M such that
\begin{align*}
	TM = D \oplus D^{\perp},
\end{align*}
 where $D$ is invariant and $D^{\perp}$ is anti-invariant.\\
 (5) A submanifold $M$ of an almost Hermitian manifold $\bar{M}$ is said to be semi-slant \cite{papaghuic}, if there exist two orthogonal complementary distributions $D$ and $D_{\theta}$ on $M$ such that
 \begin{align*}
 	TM = D \oplus D_{\theta},
 \end{align*}
 where $D$ is invariant and $D_{\theta}$ is slant with slant angle $\theta$. In this case, the angle $\theta$ is called semi-slant angle.\\
 (6) A submanifold $M$ of an almost Hermitian manifold $\bar{M}$ is said to be hemi-slant (see, \cite{Lone,stonvic}), if there
 exist two orthogonal complementary distributions $D_{\theta}$ and $D^{\perp}$ on $M$ such that
 \begin{align*}
 	TM = D_{\theta} \oplus D^{\perp},
 \end{align*}
 where $D_{\theta}$ is slant with slant angle $\theta$ and $D^{\perp}$ is anti-invariant. In this case, the angle $\theta$ is called hemi-slant angle.
 \begin{definition}
 	Let $M$ be a submanifold of an almost Hermitian manifold $\bar{M}$. Then, we say $M$ is a bi-slant submanifold of $\bar{M}$ if there exists a pair of orthogonal distributions $D_1$ and $D_2$ of $M$, at a point $ p \in M$ such that \\
 	(a) $ TM = D_1 \oplus D_2$;\\
 	(b) $ JD_1 \perp D_2$ and $JD_2 \perp D_1$;\\
 	(c) The distributions $D_1, D_2$ are pointwise slant with slant functions $\theta_1 , \theta_2 $, respectively.
 \end{definition}
 	\par The pair $\{ \theta_1, \theta_2\}$ of slant functions  is called the bi-slant function. A pointwise bi-slant submanifold $M$ is called proper if its bi-slant function satisfies $ \theta_1, \theta_2 \neq 0, \frac{\pi}{2}$ and both $\theta_1, \theta_2$ are not constant on $M$.

\section{Quasi bi-slant submanifolds of Kaehler manifolds}
In this section, we define and study quasi bi-slant submanifolds of Kaehler manifolds.
\begin{definition}
	A submanifold $M$ of an almost Hermitian manifold $\bar{M}$ is called a quasi bi-slant submanifold if
	there exist distributions $D$, $D_1$ and $D_2$ such that:\\
	(a) $TM$ admits the orthogonal direct decomposition as
	\begin{align*}
		TM = D \oplus D_1 \oplus D_2;
	\end{align*}
	(b) $J(D) = D$ i.e., $D$ is invariant;\\
	(c) $J(D_1) \perp  D_2$;\\
	(d) For any non-zero vector field $X \in (D_1)_x$; $x \in M$; the angle $\theta_1$ between $JX$ and $(D_1)_x$ is constant and independent of the choice of point $x$ and $X$ in $(D_1)_x$;\\
	(e)  For any non-zero vector field $Z \in (D_2)_y$; $y \in M$; the angle $\theta_2$ between $JZ$ and $(D_2)_y$ is constant and independent of the choice of point $y$ and $Z$ in $(D_2)_y$;\\
	The angles $\theta_1$ and $\theta_1$ are called slant angles of quasi bi-slant submanifold.
\end{definition}
\begin{remark}
	We can generalize the above definition by taking $TM = D \oplus D_{\theta_1} \oplus D_{\theta_2} ...\oplus D_{\theta_n}$. Hence we can define multi-slant submanifolds, quasi multi-slant submanifolds etc.
\end{remark}
Let $M$ be a quasi bi-slant submanifold of an almost Hermitian manifold $\bar{M}$. We denote the projections
of $X \in \Gamma(TM)$ on the distributions $D$, $D_1$ and $D_2$ by $P$, $Q$ and $R$, respectively. Then we can write, for any $X \in \Gamma(TM)$
\begin{align}\label{3.1}
	X = PX + QX + RX,
\end{align} 
we, can write
\begin{align}\label{3.2}
	JX = \phi X + \omega X,
\end{align}
where $\phi X$ and $\omega X$ are tangential and normal components of $JX$ on $M$, respectively.\\
Using (\ref{3.1}) and (\ref{3.2}), we obtain
\begin{eqnarray}\label{3.3}
	JX &=& JPX + JQX + JRX\nonumber \\ &=& \phi PX + \omega PX +\phi QX + \omega QX + \phi RX + \omega RX.
\end{eqnarray}
Since $JD = D$, we have $\omega PX = 0$. Therefore, we get
\begin{align}\label{3.4}
	JX = \phi PX +\phi QX + \omega QX + \phi RX + \omega RX.
\end{align}
This means, for any $X \in \Gamma(TM)$, we have
\begin{align*}
	\phi X = \phi PX + \phi QX + \phi RX \ \ \textnormal{and} \ \  \omega X = \omega QX + \omega RX.
\end{align*}
Thus, we have the following decomposition
\begin{align*}
	J(TM) \subset D \oplus \phi D_1 \oplus \omega D_1 \oplus \phi D_2 \oplus \omega D_2.
\end{align*}
Since $\omega D_1 \in (T^{\perp}M)$ and $\omega D_2 \in (T^{\perp}M)$, we have 
\begin{align}
		T^{\perp}M = \omega D_1 \oplus \omega D_2 \oplus \mu,
\end{align}
where $\mu$ is the orthogonal complement of $\omega D_1 \oplus \omega D_2$ in $(T^{\perp}M)$ and it is invariant with respect to $J$. \\
For any $Z \in \Gamma(T^{\perp}M)$, we put
\begin{align*}
	JZ = BZ + CZ,
\end{align*}
where $BZ \in \Gamma(TM)$ and $CZ \in \Gamma(T^{\perp}M)$.\\
\begin{lemma}\cite{Prasad1}
	Let $M$ be a quasi bi-slant submanifold of an almost Hermitian manifold $\bar{M}$, Then
\par (i) $\phi^2X = -(\cos^2\theta_1)X$,
\par (ii) $g(\phi X,\phi Y) = (\cos^2\theta_1)g(X,Y)$,
\par (iii) $ g(\omega X,\omega Y)= (\sin^2\theta_1)g(X,Y)$\\
for any $X, Y \in \Gamma(D1)$, where $\theta_1$ denotes the slant angle of $D_1$.  
\end{lemma}
\begin{lemma}\cite{Prasad1}
	Let $M$ be a quasi bi-slant submanifold of an almost Hermitian manifold $\bar{M}$, Then
	\par (i) $\phi^2Z = -(\cos^2\theta_2)Z$,
	\par (ii) $g(\phi Z,\phi W) = (\cos^2\theta_2)g(Z,W)$,
	\par (iii) $ g(\omega Z,\omega W)= (\sin^2\theta_2)g(Z,W)$\\
	for any $Z, W \in \Gamma(D2)$, where $\theta_2$ denotes the slant angle of $D_2$.  
\end{lemma}
\section{Some Results on quasi bi-slant submanifolds}
For a proper quasi bi-slant submanifold $M$ of a kaehler manifold $\bar{M}$, the normal bundle of $M$ is decomposed as 
\begin{align}\label{4.1}
	T^{\perp}M = \omega D_1 \oplus \omega D_2 \oplus \mu,
\end{align}
where $\mu$ is the orthogonal complement of $\omega D_1 \oplus \omega D_2$ in $(T^{\perp}M)$ and it is invariant with respect to $J$.
\par The following results for proper quasi bi-slant submanifolds is given as:
\begin{proposition}
	Let $M$ be a proper quasi bi-slant submanifold of a Kaehler manifold $\bar{M}$. Then, we have
	\begin{footnotesize}
	\begin{eqnarray}\label{4.2}
	g(\nabla_XY,\phi Z) &=& (\cos^2\theta_1)g(\nabla_XQY,\phi Z) + g(A_{\omega \phi QY}\phi Z,X)- (\sin^2\theta_1)g(A_{\omega Z}X,QY)\nonumber \\&& + (\cos^2\theta_2)g(\nabla_XRY,\phi Z)+ g(A_{\omega \phi RY}\phi Z,X)-(\sin^2\theta_2)g(A_{\omega Z}X,RY) \nonumber \\ &&  + g(\nabla^{\perp}_X\omega Z,\omega \phi QY) + g(\nabla^{\perp}_X\omega Z,\omega \phi RY) +  g(A_{\omega QY}Z,X)\nonumber \\&& + g(A_{\omega RY}Z,X),
	\end{eqnarray}
\end{footnotesize}
for any $X \in \Gamma(D_1)$ and $Y,Z \in \Gamma(D_1 \oplus D_2)$, where $\theta_1$ and $\theta_2$ are slant angles of slant distribution $D_1$ and $D_2$, respectively. 
\end{proposition}
\begin{proof}
	For any $X \in \Gamma(D_1)$ and $Y,Z \in \Gamma(D_1 \oplus D_2)$, we have
	\begin{eqnarray*}
		g(\nabla_XY,\phi Z) = g(\bar{\nabla}_X(QY + RY), JZ) - g(\bar{\nabla}_X(QY + RY), \omega Z).
	\end{eqnarray*}
Using (\ref{2.1}), (\ref{2.2}), (\ref{2.4}) and (\ref{3.2}), we have
\begin{eqnarray*}
 g(\nabla_XY,\phi Z) &=&	- g(J\bar{\nabla}_XQY,Z) - g(J\bar{\nabla}_XRY,Z) - g(\sigma(X,QY), \omega Z)\\ && - g(\sigma(X,RY),\omega Z).
\end{eqnarray*} 
Then by applying $\bar{\nabla} J=0$, and  using (\ref{2.6}) and (\ref{3.2}), we obtain
\begin{eqnarray*}
 g(\nabla_XY,\phi Z) &=& -g(\bar{\nabla}_XJQY,Z) - g(\bar{\nabla}_XJRY,Z) - g(\sigma(X,QY), \omega Z)\\&& - g(\sigma(X,RY),\omega Z)\\ &=& -g(\bar{\nabla}_X\phi QY,Z) -g(\bar{\nabla}_X\omega QY,Z) -g(\bar{\nabla}_X\phi RY,Z) \\&& -g(\bar{\nabla}_X\omega RY,Z) - g(A_{\omega Z}X,QY) - g(A_{\omega Z}X,RY).	
\end{eqnarray*}
On simplifying above equation and using (\ref{2.5}), we arrive at
\begin{eqnarray*}
g(\nabla_XY,\phi Z) &=& -g(\bar{\nabla}_X\phi QY,Z) + g(A_{\omega QY}X,Z) - g(\bar{\nabla}_X\phi RY,Z)\\ && + g(A_{\omega RY}X,Z)- g(A_{\omega Z}X,QY) - g(A_{\omega Z}X,RY). 
\end{eqnarray*}
Using (\ref{2.2}), the above equation can be re-written as
\begin{eqnarray*}
	g(\nabla_XY,\phi Z) &=& -g(J\bar{\nabla}_X\phi QY,JZ) + g(A_{\omega QY}X,Z) -g(J\bar{\nabla}_X\phi RY,JZ) \\ && +  g(A_{\omega RY}X,Z) - g(A_{\omega Z}X,QY) - g(A_{\omega Z}X,RY). 
\end{eqnarray*}
Since, the shape operator $A$ is self-adjoint, it follows from (\ref{2.3}), (\ref{2.6}) and (\ref{3.2}), we get
\begin{footnotesize}
\begin{eqnarray*}
	g(\nabla_XY,\phi Z) &=& -g(\bar{\nabla}_XJ\phi QY,\phi Z) - g(\bar{\nabla}_XJ\phi QY,\omega Z) -g(\bar{\nabla}_XJ\phi RY,\phi Z)\\&& - g(\bar{\nabla}_XJ\phi RY,\omega Z)  + g(A_{\omega QY}Z,X) + g(A_{\omega RY}Z,X)\\ &&- g(A_{\omega Z}X,QY) - g(A_{\omega Z}X,RY) \\ &=& -g(\bar{\nabla}_X\phi^2QY,\phi Z) - g(\bar{\nabla}_X\omega \phi QY,\phi Z) - g(\bar{\nabla}_X\phi^2RY,\phi Z)\\ && -g(\bar{\nabla}_X\omega \phi RY,\phi Z)  -g(\bar{\nabla}_X\phi^2QY,\omega Z) - g(\bar{\nabla}_X\omega \phi QY,\omega Z)\\&& -g(\bar{\nabla}_X\phi^2RY,\omega Z)- g(\bar{\nabla}_X\omega \phi RY,\omega Z) + g(A_{\omega QY}Z,X)\\&&  + g(A_{\omega RY}Z,X) - g(A_{\omega Z}X,QY) - g(A_{\omega Z}X,RY). 
\end{eqnarray*}
\end{footnotesize}
Now, using (\ref{2.5}), Lemma 3.3 and Lemma 3.4, we obtain
\begin{footnotesize}
\begin{eqnarray*}
g(\nabla_XY,\phi Z) &=& (\cos^2\theta_1)g(\nabla_XQY,\phi Z) + g(A_{\omega \phi QY}X,\phi Z) + (\cos^2\theta_2)g(\nabla_XRY,\phi Z) \\ &&+ g(A_{\omega \phi RY}X,\phi Z) + (\cos^2\theta_1)g(\bar{\nabla}_XQY,\omega Z) + g(\omega \phi QY,\bar{\nabla}_X\omega Z) \\&&(\cos^2\theta_2)g(\bar{\nabla}_XRY,\omega Z)+ g(\omega \phi RY,\bar{\nabla}_X\omega Z) +  g(A_{\omega QY}Z,X)\\&& + g(A_{\omega RY}Z,X) - g(A_{\omega Z}X,QY)- g(A_{\omega Z}X,RY).
\end{eqnarray*}
\end{footnotesize}
Again using (\ref{2.3}), (\ref{2.5}) and (\ref{2.6}), we arrive at 
\begin{footnotesize}
\begin{eqnarray*}
	g(\nabla_XY,\phi Z) &=& (\cos^2\theta_1)g(\nabla_XQY,\phi Z) + g(A_{\omega \phi QY}\phi Z,X) + (\cos^2\theta_2)g(\nabla_XRY,\phi Z) \\ &&+ g(A_{\omega \phi RY}\phi Z,X) + (\cos^2\theta_1)g(A_{\omega Z}X,QY) + g(\nabla^{\perp}_X\omega Z,\omega \phi QY) \\&&(\cos^2\theta_2)g(A_{\omega Z}X,RY)+ g(\nabla^{\perp}_X\omega Z,\omega \phi RY) +  g(A_{\omega QY}Z,X)\\&& + g(A_{\omega RY}Z,X) - g(A_{\omega Z}X,QY)- g(A_{\omega Z}X,RY).
\end{eqnarray*}
\end{footnotesize}
Now, from above relation, the desired result follows. Hence, the proof is complete.
\end{proof}
\begin{proposition}
	Let $M$ be a proper quasi bi-slant submanifold of a Kaehler manifold $\bar{M}$. Then, we have
\begin{small}
	\begin{eqnarray}\label{4.3}
	g([Y,Z],\phi X) &=&  g(A_{\omega \phi Z}\phi X,Y) - g(A_{\omega \phi Y}\phi X,Z) + g(\nabla^{\perp}_Y\omega X,\omega \phi Z) \\&&- g(\nabla^{\perp}_Z\omega X,\omega \phi Y) +  g(A_{\omega Z}X,Y)\nonumber - g(A_{\omega Y}X,Z)
\end{eqnarray}
\end{small}
for any $X \in \Gamma(D_1)$ and $Y,Z \in \Gamma(D_1 \oplus D_2)$, where $\theta_1$ and $\theta_2$ are slant angles of slant distribution $D_1$ and $D_2$, respectively. 
\end{proposition}
\begin{proof}
 In a similar fashion, as Proposition 4.1, we can derive
 	\begin{footnotesize}
 	\begin{eqnarray}\label{4.4}
 		g(\nabla_ZY,\phi X) &=& (\cos^2\theta_1)g(\nabla_{QZ}QY,\phi X) + g(A_{\omega \phi QY}\phi X,QZ)- (\sin^2\theta_1)g(A_{\omega X}QZ,QY)\nonumber \\&& + (\cos^2\theta_2)g(\nabla_{RZ}RY,\phi X)+ g(A_{\omega \phi RY}\phi X,RZ)-(\sin^2\theta_2)g(A_{\omega X}RZ,RY) \nonumber \\ &&  + g(\nabla^{\perp}_Z\omega X,\omega \phi QY) + g(\nabla^{\perp}_Z\omega X,\omega \phi RY) +  g(A_{\omega QY}X,Z)\nonumber \\&& + g(A_{\omega RY}X,Z),
 	\end{eqnarray}
 \end{footnotesize}
for $X \in \Gamma(D)$ and $Y,Z \in \Gamma(D_1 \oplus D_2)$.
\par Interchanging $Y$ and $Z$ in (\ref{4.4}) yields
	\begin{footnotesize}
	\begin{eqnarray}\label{4.5}
		g(\nabla_YZ,\phi X) &=& (\cos^2\theta_1)g(\nabla_{QY}QZ,\phi X) + g(A_{\omega \phi QZ}\phi X,QY)- (\sin^2\theta_1)g(A_{\omega X}QY,QZ)\nonumber \\&& + (\cos^2\theta_2)g(\nabla_{RY}RZ,\phi X)+ g(A_{\omega \phi RZ}\phi X,RY)-(\sin^2\theta_2)g(A_{\omega X}RY,RZ) \nonumber \\ &&  + g(\nabla^{\perp}_Y\omega X,\omega \phi QZ) + g(\nabla^{\perp}_Y\omega X,\omega \phi RZ) +  g(A_{\omega QZ}X,Y)\nonumber \\&& + g(A_{\omega RZ}X,Y).
	\end{eqnarray}
\end{footnotesize}
Then after using symmetry of shape operator, (\ref{3.1}) and subtracting (\ref{4.4}) from (\ref{4.5}) , we obtain (\ref{4.3}). Hence the proof is complete.
	
\end{proof}
\section{Warped product quasi bi-slant submanifolds of Kaehler manifold}
Let $(M_1, g_1)$ and $(M_2, g_2)$ be two Riemannian manifolds and $f > 0$, be a positive differentiable function on $M_1$. Consider the product	manifold $M_1 \times M_2$ with its canonical projections $\pi : M_1 \times M_2 \rightarrow M_1$ and $\rho : M_1 \times M_2 \rightarrow M_2$. The warped product $M = M_1 \times_f M_2$ is the product manifold $M_1 \times M_2$ equipped with the Riemannian metric $g$ such that\\
\begin{eqnarray*}
	g(X,Y) = g_1(\pi_\ast(X),\pi_\ast(Y)) + (f \circ \pi)^2g_2(\rho_\ast(X),\rho_\ast(Y))
\end{eqnarray*}
for any tangent vector $X, Y \in TM$, where $\ast$ is the symbol for the tangent maps. It was proved in \cite{Neill} that for any $ X \in TM_1$ and $Z \in TM_2$, the following holds
\begin{eqnarray}\label{5.1}
	\nabla_XZ = \nabla_ZX = (Xln f)Z
\end{eqnarray}
where $\nabla$ denotes the Levi-Civita connection of $g$ on $M$. A warped product manifold $ M = M_1 \times_f M_2$ is said to be trivial if the warping function $f$ is constant. If $M = M_1 \times_f M_2$ is a warped product manifold then $M_1$ is totally geodesic and $M_2$ is a totally umbilical (see \cite{Neill,6}).\\
From now onwards, we assume the ambient manifold $\bar{M}$ is Kaehler manifold and $M$ is quasi bi-slant submanifold in $\bar{M}$.\\
Now we give the following useful lemma for later use.
\begin{lemma}
	Let $M = M_1 \times_fM_2$, where $M_2 = M_{\theta_1} \times M_{\theta_2}$ be a warped product quasi bi-slant submanifold of a Kaehler manifold $\bar{M}$. Then, we have
	\begin{align}
		g(\sigma(X,Y),\omega Z) = g(\sigma(X,Z),\omega QY) + g(\sigma(X,Z),\omega RY),
	\end{align}
for any $X, Z \in \Gamma(M_1)$ and $Y \in \Gamma(M_2)$.
\end{lemma}
\begin{proof}
	For any  $X,Y \in \Gamma(M_1)$ and $Z \in \Gamma(M_2)$, we have
	\begin{eqnarray*}
		g(\sigma(X,Y),\omega Z) &=& g(\bar{\nabla}_XY,\omega Z)\\ &=& g(\bar{\nabla}_XY,JZ) - g(\bar{\nabla}_XY,\phi Z).
	\end{eqnarray*}
Using (\ref{2.2}), (\ref{2.3}) and the orthogonality of vector fields given in condition $(c)$ of Definition 3.1, we find
\begin{eqnarray*}
		g(\sigma(X,Y),\omega Z) = -g(\bar{\nabla}_XJY,Z) + g(\bar{\nabla}_X \phi Z,Y).
\end{eqnarray*}
Using (\ref{2.4}), (\ref{3.1}), (\ref{3.2}), (\ref{5.1}) and orthoganality of vector fields, we obtain
\begin{eqnarray*}
		g(\sigma(X,Y),\omega Z) &=& - g(\bar{\nabla}_X\phi QY,Z) - g(\bar{\nabla}_X\omega QY,Z) - g(\bar{\nabla}_X\phi RY,Z)\\&& - g(\bar{\nabla}_X\omega RY,Z) -(X ln f)g(\phi Z,Y).\\&=& -g(\bar{\nabla}_X\phi Y,Z) - g(\bar{\nabla}_X\omega QY,Z) - g(\bar{\nabla}_X\omega RY,Z). 
\end{eqnarray*}
Thus, using (\ref{2.5}) and (\ref{5.1}), we arrive at
\begin{eqnarray*}
		g(\sigma(X,Y),\omega Z) &=& g(\nabla_XZ,\phi Y) + g(A_{\omega QY}X,Z) + g(A_{\omega RY}X,Z) \\ &=& 
		(X ln f)g(Z, \phi Y) + g(\sigma(X,Z),\omega QY) + g(\sigma(X,Z),\omega RY).
\end{eqnarray*}
Using the orthoganality of vector fields, we get 
\begin{align*}
		g(\sigma(X,Y),\omega Z) =  g(\sigma(X,Z),\omega QY) + g(\sigma(X,Z),\omega RY).
\end{align*}
Hence, the proof of lemma follows.
\end{proof}
\begin{lemma}
Let $M = M_1 \times_fM_2$, where $M_2 = M_{\theta_1} \times M_{\theta_2}$ be a warped product quasi bi-slant submanifold of a Kaehler manifold $\bar{M}$. Then, we have
	\begin{align}
	g(\sigma(X,Z),\omega W) = g(\sigma(X,W),\omega QZ) + g(\sigma(X,W),\omega RZ),
\end{align}
	for any  $X\in \Gamma(M_1)$ and $Z, W \in \Gamma(M_2)$.
\end{lemma}
\begin{proof}
	For any  $X \in \Gamma(M_1)$ and $Z, W \in \Gamma(M_2)$, we have
\begin{eqnarray*}
	g(\sigma(X,Z),\omega W)&=&g(\bar{\nabla}_XZ,JW) - g(\nabla_XZ,\phi W)\\ &=& - g(\bar{\nabla}_XJZ,W) -  g(\nabla_XZ,\phi W).	
\end{eqnarray*} 
Using (\ref{2.4}), (\ref{3.1}) and (\ref{3.2}), we obtain
\begin{eqnarray*}
		g(\sigma(X,Z),\omega W)&=&- g(\bar{\nabla}_X\phi QZ,W) - g(\bar{\nabla}_X\omega QZ,W) - g(\bar{\nabla}_X\phi RZ,W) \\&&- g(\bar{\nabla}_X\omega RZ,W) - g(\nabla_XZ,\phi W). 
\end{eqnarray*}
On further simplification, we arrive at
\begin{small} 
\begin{eqnarray*}
		g(\sigma(X,Z),\omega W)&=& -(\phi QZ ln f)g(X,W) + g(A_{\omega QZ}X,W) - (\phi RZ ln f)g(X,W) \\&&+ g(A_{\omega RZ}X,W) - (Z ln f)g(X,W).
\end{eqnarray*}
\end{small}
Using orthogonality of vector fields, we get
\begin{align*}
		g(\sigma(X,Z),\omega W) =  g(A_{\omega QZ}X,W) +  g(A_{\omega RZ}X,W).
\end{align*}
Using (\ref{2.6}), we have 
\begin{align*}
	g(\sigma(X,Z),\omega W) = g(\sigma(X,W),\omega QZ) + g(\sigma(X,W),\omega RZ). 
\end{align*}
Hence, the proof of lemma is complete.
\end{proof}
\section{Main Result}
\begin{theorem}
	Let $M = M_1 \times_f M_2$ where $M_2 = M_{\theta_1} \times M_{\theta_2}$ be a warped product quasi bi-slant submanifold with bi-slant angles $\{ \theta_1, \theta_2\}$ in a Kaehler manifold $\bar{M}$. Then one of the following two cases must occur: 
\par	(1) The warping function $f$ is constant i.e., $M$ is Riemannian product;
\par    (2) $ \theta_2 = \frac{\pi}{2}$, i.e., $M$ is a warped product quasi hemi-slant submanifold such that $M_{\theta_2}$ is a totally real submanifold $M_{\perp}$ of $\bar{M}$.
\end{theorem}
\begin{proof}
	Let $M = M_1 \times_f M_2$ where $M_2 = M_{\theta_1} \times M_{\theta_2}$ be a warped product quasi bi-slant submanifold of a Kaehler manifold $\bar{M}$ with bi-slant angles $\{\theta_1,\theta_2\}$. Then, we have
	\begin{align}\label{6.1}
		g(\sigma(X,Z),\omega W) = g(\bar{\nabla}_ZX,JW) - g(\nabla_ZX,\phi W),
	\end{align}
for any $X \in \Gamma(M_1)$ and $Z,W \in \Gamma(M_2)$. Thus, by using (\ref{2.1})-(\ref{2.4}) and (\ref{5.1}), we obtain
\begin{eqnarray*}
	g(\sigma(X,Z),\omega W) &=& -g(\bar{\nabla}_ZJX,W) - (X ln f)g(Z,\phi W) \\ &=& -g(\bar{\nabla}_Z\phi PX,W) - g(\bar{\nabla}_Z\phi QX,W) - g(\bar{\nabla}_Z\omega QX,W) \\ && - g(\bar{\nabla}_Z\phi RX,W) - g(\bar{\nabla}_Z\omega RX,W) - (X ln f)g(Z,\phi W).  
\end{eqnarray*}
Thus, it follows from (\ref{2.5}) and (\ref{5.1})
\begin{eqnarray*}
g(\sigma(X,Z),\omega W) &=& -(\phi PX ln f)g(Z,W) - (\phi QX ln f)g(Z,W) \\ && + g(A_{\omega QX}Z,W) - (\phi RX ln f)g(Z,W) \\ && + g(A_{\omega RX}Z,W) - (X ln f)g(Z,\phi W).	
\end{eqnarray*}
Using (\ref{2.6}), we arrive at
\begin{eqnarray}\label{6.2}
	g(\sigma(X,Z),\omega W) &=& -(\phi PX ln f)g(Z,W) - (\phi QX ln f)g(Z,W)\nonumber \\ && + g(\sigma(Z,W), \omega QX) - (\phi RX ln f)g(Z,W) \nonumber\\ && +  g(\sigma(Z,W), \omega RX) - (X ln f)g(Z,\phi W).	
\end{eqnarray}
Interchanging $Z$ by $W$ in (\ref{6.2}) and using (\ref{2.2}), we get
\begin{eqnarray}\label{6.3}
	g(\sigma(X,W),\omega Z) &=& -(\phi PX ln f)g(Z,W) - (\phi QX ln f)g(Z,W)\nonumber \\ && + g(\sigma(Z,W), \omega QX) - (\phi RX ln f)g(Z,W)\nonumber \\ && +  g(\sigma(Z,W), \omega RX) + (X ln f)g(Z,\phi W).	
\end{eqnarray}
Subtracting (\ref{6.2}) from (\ref{6.3}) and by applying Lemma 5.2 , we arrive at 
\begin{align}\label{6.4}
	(X ln f)g(Z,\phi W) = 0.
\end{align}
Again, after interchanging $W$ by $\phi W$ in (\ref{6.4}) and using Lemma 3.4, we get
\begin{align*}
	(\cos^2\theta_2)(X ln f)g(Z,W) = 0.
\end{align*}
Therefore, either $f$ is constant or $\cos \theta_2 = 0$ holds. Consequently, either $M$ is a Riemannian product manifold or $\theta_2 = \frac{\pi}{2}$. In the second case, $M$ is a warped product quasi hemi-slant submanifold which hasbeen studied in \cite{prasad2}.
\end{proof}

\section{Some examples on warped product quasi bi-slant submanifolds of Kaehler manifold.}
\begin{example}
	Let $\mathbb{E}^{2n}$ be the Euclidean $2n$-space with the standard metric and let $\mathbb{C}^n$ denotes the complex Euclidean $n$-space $(\mathbb{E}^{2n},J)$ equipped with the canonical complex structure $J$ defined as 
	\begin{align*}
		J\bigg(\frac{\partial}{\partial x_i}\bigg) = \frac{\partial}{\partial y_i}, \hspace*{0.5cm} 	J\bigg(\frac{\partial}{\partial y_j}\bigg) = -\frac{\partial}{\partial x_j}, 1\leq i,j \leq n. 
	\end{align*}
Consider a submanifold $M$ of $\mathbb{C}^6$ defined by
\begin{eqnarray*}
\chi (u,v,w,r,s,t) &=& (u\cos \theta_1, v\cos \theta_1, u\sin \theta_1, v\sin \theta_1, w\cos \theta_2, r\cos \theta_2,\\&&  w\sin \theta_2,  r\sin \theta_2, -u-w+v+r, u+w+v+r, s, t ).	
\end{eqnarray*}
It is easy to see that the tangent bundle $TM$ of $M$ is spanned by the following vectors
\begin{eqnarray*}
	Z_1 = \cos \theta_1 \frac{\partial}{\partial x_1} + \sin \theta_1 \frac{\partial}{\partial x_2} - \frac{\partial}{\partial x_5} + \frac{\partial}{\partial y_5},
\end{eqnarray*}
\begin{eqnarray*}
	Z_2 = \cos \theta_1 \frac{\partial}{\partial y_1} + \sin \theta_1 \frac{\partial}{\partial y_2} + \frac{\partial}{\partial x_5} + \frac{\partial}{\partial y_5},
\end{eqnarray*}
\begin{eqnarray*}
	Z_3 = \cos \theta_2 \frac{\partial}{\partial x_3} + \sin \theta_2 \frac{\partial}{\partial x_4} - \frac{\partial}{\partial x_5} +  \frac{\partial}{\partial y_5},
\end{eqnarray*}
\begin{eqnarray*}
	Z_4 = \cos \theta_2 \frac{\partial}{\partial y_3} + \sin \theta_2 \frac{\partial}{\partial y_4} + \frac{\partial}{\partial x_5} +  \frac{\partial}{\partial y_5},
\end{eqnarray*}
\begin{eqnarray*}
	Z_5 = \frac{\partial}{\partial x_6}, \hspace*{0.5cm} Z_6 = \frac{\partial}{\partial y_6}.
\end{eqnarray*}
Then , clearly we obtain
\begin{eqnarray*}
	JZ_1 =  \cos \theta_1 \frac{\partial}{\partial y_1} + \sin \theta_1 \frac{\partial}{\partial y_2} - \frac{\partial}{\partial y_5} - \frac{\partial}{\partial x_5},
\end{eqnarray*}
\begin{eqnarray*}
	JZ_2 = -\cos \theta_1 \frac{\partial}{\partial x_1} - \sin \theta_1 \frac{\partial}{\partial x_2} + \frac{\partial}{\partial y_5} - \frac{\partial}{\partial x_5},
\end{eqnarray*}
\begin{eqnarray*}
	JZ_3 = \cos \theta_2 \frac{\partial}{\partial y_3} + \sin \theta_2 \frac{\partial}{\partial y_4} - \frac{\partial}{\partial y_5} - \frac{\partial}{\partial x_5},
\end{eqnarray*}
\begin{eqnarray*}
	JZ_4 = - \cos \theta_2 \frac{\partial}{\partial x_3} - \sin \theta_2 \frac{\partial}{\partial x_4} + \frac{\partial}{\partial y_5} - \frac{\partial}{\partial x_5},
\end{eqnarray*}
\begin{eqnarray*}
	JZ_5 = \frac{\partial}{\partial y_6}, \hspace*{0.5cm} JZ_6 = - \frac{\partial}{\partial x_6}.
\end{eqnarray*}
Then, we find that $D = span\{Z_5,Z_6\}$ is an invariant distribution, $D_1 = span\{Z_1,Z_2\}$ is a proper  slant distribution with slant angle $\theta_1 = \cos^{-1}(\frac{1}{3})$ and $D_2 = span \{Z_3,Z_4\}$ is again a proper slant distribution with slant angle $\theta_2 = \cos^{-1}(\frac{1}{3})$. Hence the submanifold $M$ defined by $\chi$ is a proper quasi bi-slant submanifold of $\mathbb{C}^6$.
\par It is easy to verify that $D$ and $D_1 \oplus D_2$ are integrable. If we denote the integrable manifolds of $D$, $D_1$ and $D_2$ by $M_T$, $M_{\theta_1}$ and $M_{\theta_2}$, respectively. Then the metric tensor $g$ of product manifold $M$ is given by 
\begin{eqnarray*}
	ds^2 &=& ds^2 + dt^2 + 3(du^2 + dv^2) + 3(dw^2 + dr^2)\\ &=& g_{M_T} + 3g_{M_2},
\end{eqnarray*}
such that,
\begin{align*}
	g_{M_T} = ds^2 + dt^2 \hspace{0.5cm} and \hspace{0.5cm} g_{M_2} = 3(du^2 + dv^2) + dw^2 + dr^2,
\end{align*}
where $M_2 = M_{\theta_1} \times M_{\theta_2}$. In this case the warping function $f= \sqrt{3}$ a constant, and hence $M$ is simply a Riemannian product.

\end{example}
\begin{example}
 Consider a submanifold $M$ of $\mathbb{C}^5$ defined by 
	\begin{align*}
		\chi(u,v,w,s,t) = (v\cos u, w\cos u, v\sin u, w\sin u, -v+w, v+w, 0, 0, s, t),
	\end{align*}
	with almost complex structure $J$ defined by 
\begin{align*}
	J\bigg(\frac{\partial}{\partial x_i}\bigg) = \frac{\partial}{\partial y_i}, \hspace*{0.5cm} 	J\bigg(\frac{\partial}{\partial y_j}\bigg) = -\frac{\partial}{\partial x_j}, 1\leq i,j \leq 5. 
\end{align*}
	It is easy to see that its tangent space $TM$ of $M$ is spanned by the following vectors
	\begin{align*}
		v_1 = -v\sin u \frac{\partial}{\partial x_1} + v\cos u \frac{\partial}{\partial x_2} - w\sin u\frac{\partial}{\partial y_1} + w\cos u \frac{\partial}{\partial y_2},
	\end{align*}
	\begin{align*}
	v_2 = \cos u \frac{\partial}{\partial x_1} + \sin u  \frac{\partial}{\partial x_2} - \frac{\partial}{\partial x_3} + \frac{\partial}{\partial y_3},
\end{align*}
\begin{align*}
	v_3 = \frac{\partial}{\partial x_3} + \cos u \frac{\partial}{\partial y_1} + \sin u \frac{\partial}{\partial y_2} + \frac{\partial}{\partial y_3},
	\end{align*}
	\begin{align*}
		v_4 = \frac{\partial}{\partial x_5}, \hspace*{0.5cm} v_6 = \frac{\partial}{\partial y_5}.
	\end{align*}

Then, we have 
	\begin{align*}
	Jv_1 = -v\sin u \frac{\partial}{\partial y_1} + v \cos u \frac{\partial}{\partial y_2} + w \sin u \frac{\partial}{\partial x_1} - w \cos u  \frac{\partial}{\partial x_2},
\end{align*}
\begin{align*}
	Jv_2 = \cos u \frac{\partial}{\partial y_1} + \sin u  \frac{\partial}{\partial y_2} -  \frac{\partial}{\partial y_3} -  \frac{\partial}{\partial x_3},
\end{align*}
\begin{align*}
	Jv_3 = \frac{\partial}{\partial y_3} - \cos u \frac{\partial}{\partial x_1} - \sin u  \frac{\partial}{\partial x_2 } - \frac{\partial}{\partial x_3}.
\end{align*}
\begin{align*}
		Jv_4 = \frac{\partial}{\partial y_5}, \hspace*{0.5cm} Jv_5 = - \frac{\partial}{\partial x_5}.
\end{align*}
Let us put $D = span\{v_4,v_5\}$ is an invariant distribution, $D_1 = span \{v_2,v_3\}$ a proper slant distribution with slant angle $\theta_1 = \cos^{-1}\big(\frac{1}{3}\big)$ and $D_2 = span \{v_1\}$ an anti-invariant distribution with slant angle $\theta_2 = \frac{\pi}{2}$. Hence the submanifold $M$ defined by $\chi$ is a quasi bi-slant submanifold of $\mathbb{C}^5$.
\par It is easy to verify that $D$ and $D_1 \oplus D_2$ are integrable. If we denote the integrable manifolds of $D$, $D_1$ and $D_2$ by $M_T$, $M_{\theta_1}$ and $M_{\perp}$, respectively. Then the metric tensor $g$ of product manifold $M$ is given by 
\begin{align*}
	ds^2 = ds^2 + dt^2 + 3(dv^2 + dw^2) + (v^2 + w^2)du^2,
\end{align*}
such that
\begin{align*}
	g_{M_T} = ds^2+dt^2 \hspace{0.5cm} and \hspace{0.5cm} g_{M_2} = 3(dv^2+dw^2) + du^2,
\end{align*}
where $M_2 = M_{\theta_1} \times M_{\perp}$. In this case the warping function $f= \sqrt{v^2+w^2}$ and hence $M$ is a case of warped product quasi hemi-slant submanifold. So we have discussed both the case of Theorem 6.1.
 \end{example}

\textbf{Data Availability Statement}: The authors declare that this research is purely theoretical and does not associate with any datas.\\
\textbf{Conflicts of Interest}: The authors declare that they have no conflict of interest, regarding the publication of
this paper.

\end{document}